\documentclass[a4paper,12pt]{amsart}
\usepackage[utf8]{inputenc}
\usepackage[german, english]{babel}
\usepackage[T1]{fontenc}
\usepackage{amsmath,amssymb,amsthm}
\usepackage{enumerate}

\newcommand{\IR}{\mathbb{R}}

\newtheorem{theorem}{Theorem}

\newtheorem{prop}{Proposition}
\newtheorem{defin}[prop]{Definition}
\newtheorem{example}[prop]{Example}
\newtheorem{lemma}[prop]{Lemma}
\newtheorem{remarks}[prop]{Remarks}
\newtheorem{coro}[prop]{Corollary}

\begin{document}
\title{Nonlinear Beurling-Deny criteria}

\author[Puchert]{Simon Puchert}
\address{Institut für Mathematik \\Friedrich-Schiller-Universität Jena \\07743 Jena, Germany } \email{simon.puchert@uni-jena.de}

\begin{abstract}
 This note introduces a simple symmetric contraction property for functionals. This property clearly characterizes Dirichlet forms in the linear case. We show that it also characterizes Dirichlet forms in the non-linear case. Furthermore, we use this property to gain a new perspective on criteria of Cipriani / Grillo as well as Brigati / Hartarsky.
\end{abstract}
\maketitle

\section*{Introduction}
The study of (bilinear) Dirichlet forms goes back to the work of Beurling / Deny \cite{BD, BD1}. Starting with \cite{CG}, there has been a growing interest in a non-linear version of the theory, see e.g. \cite{CHK, CW, HKMP, W, G}. A crucial ingredient in all these considerations is compatibility of the functionals with normal contractions. This was used in \cite{CG} to characterize $L^\infty$-contractivity and order preservation of the associated semigroup.

For a bilinear functional $\mathcal{E}$ this compatibility takes the simple form 
 $$\mathcal{E}(Cf)\leq \mathcal{E}(f)$$
for all (some) normal contractions $C$.

In the non-linear case the situation is more complex. Indeed, following \cite{BP, CG}, various characterizations have been considered \cite{BP,Cla,BH}. 

Here, it is our aim to introduce a rather simple condition much in the spirit of the original work of Beurling / Deny. The key point is that it is not applied to the original functional $\mathcal{E}$ but to a derived functional $\mathcal{E}_f$ that coincides with $\mathcal{E}$ in the linear case. Our condition then just reads 
 $$ \mathcal{E}_f (Cg)\leq \mathcal{E}_f (g) $$
for all $f,g$ in the domain and all (some) normal contractions $C$.
 
Precise definitions (Definition~\ref{def-main}) are given in the next section, where also the equivalence to a condition of Bénilan / Picard \cite{BP}, Théorème~2.1 (see also Claus \cite{Cla}, Theorem~2.39) is shown (Theorem~\ref{thm-BP}). This is followed by a section giving an interpretation and new variants of the criterion of Cipriani / Grillo \cite{CG}, Definition~3.1 (Theorem~\ref{thm-CG}). The equivalence to the conditions in Brigati / Hartarsky \cite{BH}, Theorem~1.3 is shown in the last section (Theorem~\ref{thm-BH}).

The final section contains a short summary of the results.

\textbf{Acknowledgements.} The author would like to thank Daniel Lenz and Marcel Schmidt for their help in preparing the manuscript. In fact, the small gap in the proof of \cite{BH}, Theorem~1.2, was found by Marcel Schmidt. Further thanks go to Giovanni Brigati for helpful discussions.

\section{The contraction property}
Let $(X,\mathcal{A})$ be a measurable space and let $m$ be a measure on $X$. We consider functionals on the real Hilbert space $L^2(X,m)$,
 $$\mathcal{E}\colon L^2(X,m)\to [0,\infty].$$
 
The set
 $$ D(\mathcal{E}) = \{f\in L^2(X,m) \mid \mathcal{E}(f) < \infty\} $$
is called the \emph{effective domain} of $\mathcal{E}$. The functional $\mathcal{E}$ is considered to be \emph{proper}, if its effective domain is nonempty.

Additionally, $\mathcal{E}$ is called \emph{symmetric}, if for all $f\in L^2(X,m)$, we have $\mathcal{E}(-f) = \mathcal{E}(f)$.

In the following sections, we want to analyze the following definition of nonlinear Dirichlet forms.

\begin{defin}[\cite{CG}, Definition 3.1]\label{def-main}
 Let $\mathcal{E}\colon L^2(X,m) \to [0,\infty]$ be a proper, lower semicontinuous, convex functional. Then $\mathcal{E}$ is a nonlinear Dirichlet form (in the sense of Cipriani / Grillo), if for all $u,v\in L^2(X,m)$ and all $\alpha > 0$, we have
    $$ \mathcal{E}\left(\frac{u + u\wedge v}{2}\right) + \mathcal{E}\left(\frac{v + u\vee v}{2}\right) \leq \mathcal{E}(u) + \mathcal{E}(v) $$
   as well as
    $$ \mathcal{E}(P_{2,\alpha}^1(u,v)) + \mathcal{E}(P_{2,\alpha}^2(u,v)) \leq \mathcal{E}(u) + \mathcal{E}(v), $$
   where
    $$ P_{2,\alpha}^1(u,v) = v + \frac{1}{2}[(u-v+\alpha)_+ - (u-v-\alpha)_-] $$
   and
    $$ P_{2,\alpha}^2(u,v) = u - \frac{1}{2}[(u-v+\alpha)_+ - (u-v-\alpha)_-]. $$
\end{defin}

This definition is somewhat involved, motivating the search for alternative approaches. The following contraction property will lay the foundation for our subsequent characterizations.

We call a function $C\colon \IR\to \IR$ a \emph{normal contraction} if $C(0) = 0$ and $C$ is $1$-Lipschitz, i.e. $|C(x) - C(y)| \leq |x-y|$ for all $x,y\in \IR$.

\begin{defin}[The contraction property]\label{def:contraction}
 A functional $\mathcal{E}\colon L^2(X,m)\to [0,\infty]$ is called \emph{compatible} with the normal contraction $C\colon \IR \to\IR$ 
 if for all $f,g\in L^2(X,m)$ the inequality  
  $$ \mathcal{E}(f + Cg) + \mathcal{E}(f - Cg) \leq \mathcal{E}(f+g) + \mathcal{E}(f-g) $$
 holds. If $\mathcal{E}$ is compatible with all normal contractions it is said to have the \emph{contraction property}.
\end{defin}

\begin{remarks}\label{rem-contraction}
 \begin{enumerate}[(a)]
  \item A functional $\mathcal{E}$ is called \emph{bilinear}, if there exists a bilinear map $\mathcal{E}'\colon \mathcal{D}\times \mathcal{D}\to \IR$ on some subspace $\mathcal{D}$ of $L^2(X,m)$ with $\mathcal{E}'(f,f) = \mathcal{E}(f)$ for all $f\in \mathcal{D}$ and $\mathcal{E}(f) =\infty$ for all $f\notin \mathcal{D}$. 
  By the parallelogram identity, a bilinear functional $\mathcal{E}$ is compatible with the contraction $C$ if and only if $\mathcal{E}(Cf)\leq \mathcal{E}(f)$ holds for all $f\in L^2 (X,m)$.
 
  \item The contraction property more directly depends on the functionals $g\mapsto \mathcal{E}(f+g) + \mathcal{E}(f-g)$. We modify these functionals slightly in order to coincide with $\mathcal{E}$ in the bilinear case. For each $f\in D(\mathcal{E})$, we define the $f$-shift
   $$ \mathcal{E}_f\colon L^2(X,m) \to [0,\infty], \quad \mathcal{E}_f(g) := \frac{1}{2}(\mathcal{E}(f+g) + \mathcal{E}(f-g)) - \mathcal{E}(f). $$
  These $f$-shifts can be interpreted as second-order central differences. As such, the $f$-shifts are symmetric and satisfy $\mathcal{E}(0) = 0$.
  
  With this at hand, compatibility with normal contractions is just the classical contraction property for all of these functionals, i.e.
   $$ \mathcal{E}_f(Cg) \leq \mathcal{E}_f(g) $$
  for all $f\in D(\mathcal{E})$, all functions $g\in L^2(X,m)$ and all normal contractions $C$. For the remaining case $\mathcal{E}(f) = \infty$ we observe that it necessarily entails $\mathcal{E}(f+g) + \mathcal{E}(f-g) = \infty$ by convexity.

  \item For the functionals $\mathcal{E}_f$ from (b), a direct computation establishes validity of 
   $$ (\mathcal{E}_f)_g = \frac{1}{2}(\mathcal{E}_{f+g} + \mathcal{E}_{f-g}) $$
  for all $f,g$ with $\mathcal{E}(f) < \infty$ and $\mathcal{E}_f(g) < \infty$ (which is exactly where each side is defined). Therefore, the contraction property for $\mathcal{E}$ implies the contraction property for all $\mathcal{E}_f$.

  This idea allows us to reduce some properties of nonlinear Dirichlet forms to the case of symmetric nonlinear Dirichlet forms $\mathcal{E}$ satisfying $\mathcal{E}(0) = 0$.

  In theory, we can use the $f$-shifts of given nonlinear Dirichlet forms to generate new nonlinear Dirichlet forms, but there does not seem to be any useful application at the moment.
  
  Using the $f$-shifts also reduces proofs of the contraction property for classes of functionals that are closed under the shift operation to just showing the Beurling-Deny criterion, as the next example demonstrates. As seen above, the bilinear Dirichlet forms are invariant under this operation, so this would also be a (well-known) application.
 \end{enumerate}
\end{remarks}

\begin{example}
  Let $X$ be a countable set, equipped with the counting measure. The class of mixed Dirichlet energies, given by all functionals of the form
   $$ \mathcal{E}\colon \ell^2(X)\to[0,\infty],\quad \mathcal{E}(f) := \sum\limits_{x,y\in X} b_{x,y}(f(x) - f(y)) $$
  with convex, symmetric, lower semicontinuous functions $b_{x,y}\colon \IR\to [0,\infty]$ satisfying $b_{x,y}(0) = 0$ for each edge $(x,y)\in X\times X$, is closed under the shift operation.
  
  More specifically,
   $$ \mathcal{E}_f(g) = \sum\limits_{x,y\in X} \widetilde{b_{x,y}}(g(x) - g(y)) $$
  with
   $$ \widetilde{b_{x,y}}(t) = \frac{1}{2}(b_{x,y}(f(x) - f(y) + t) + b_{x,y}(f(x) - f(y) - t)) - b_{x,y}(f(x) - f(y)). $$
  
  As we saw in the previous remark, the contraction property is equivalent to $\mathcal{E}_f(Cg) \leq \mathcal{E}_f(g)$ for all $f\in D(\mathcal{E})$, all $g\in L^2(X,m)$ and all normal contractions $C$. Since $\mathcal{E}_f$ is just another mixed Dirichlet energy, it suffices to show $\mathcal{E}(Cg)\leq \mathcal{E}(g)$ for all $g\in L^2(X,m)$, all normal contractions $C$ and all mixed Dirichlet energies $\mathcal{E}$. This is trivial.
  
  Indeed, mixed Dirichlet energies are nonlinear Dirichlet forms. Despite the similarities to the classical (bilinear) Dirichlet energies on graphs, this doesn't seem to be mentioned explicitly in the literature.
\end{example}

For later use, we make the following simple observations. In addition to their structural relevance, these lemmas will allow us to reduce the set of contractions we have to consider. Lemma \ref{composition} allows a reduction to a dense subset with respect to pointwise convergence, while Lemma \ref{pw-conv} allows a reduction to generators with respect to composition.

\begin{lemma}\label{composition}
 If a functional $\mathcal{E}\colon L^2(X,m)\to [0,\infty]$ is compatible with the normal contractions $C_1, C_2 \colon \IR \to \IR$, then it is also compatible with their composition $C_1\circ C_2$.
 
 Furthermore, any functional is compatible with the normal contractions $x\mapsto x$ and $x\mapsto -x$.
\end{lemma}
\begin{proof}
 This is a simple chained application of the compatibility:
 \begin{align*}
  \mathcal{E}(f + C_1\circ C_2 g) + \mathcal{E}(f - C_1\circ C_2 g) & \leq \mathcal{E}(f + C_2 g) + \mathcal{E}(f - C_2 g) \\
                                                                    & \leq \mathcal{E}(f + g) + \mathcal{E}(f - g)
 \end{align*}
 for all $f,g\in L^2(X,m)$.
 
 The second part is trivial.
\end{proof}

\begin{lemma}\label{pw-conv}
 Let $\mathcal{E}\colon L^2(X,m)\to [0,\infty]$ be a lower semicontinuous functional and let $(C_n)_{n\in\mathbb{N}}$ be a sequence of normal contractions that has the pointwise limit $C$, i.e. for all $t\in \IR$ we have $C_n(t) \to C(t)$ as $n\to \infty$.
 
 Then, if $\mathcal{E}$ is compatible with the normal contractions $C_n$, it is also compatible with their pointwise limit $C$.
\end{lemma}
\begin{proof}
 Let $f,g \in L^2(X,m)$. Pointwise convergence of $C_n$ and $|C_n g| \leq |g| \in L^2(X,m)$ implies that $C_n g \to C g$ in $L^2(X,m)$ via Lebesgue's theorem. This means $f\pm C_n g \to f\pm C g$ in $L^2(X,m)$ and thus
 \begin{align*}
  \mathcal{E}(f + Cg) + \mathcal{E}(f - Cg) & \leq \liminf \mathcal{E}(f + C_n g) + \liminf \mathcal{E}(f - C_n g) \\
                                            & \leq \liminf \mathcal{E}(f + C_n g) + \mathcal{E}(f - C_n g) \\
                                            & \leq \mathcal{E}(f+g) + \mathcal{E}(f-g)
 \end{align*}
 by lower semicontinuity.
\end{proof}

We now discuss how the contraction property relates to the condition by Bénilan / Picard \cite{BP}, Théorème~2.1. This theorem was recently discussed and streamlined by Claus in \cite{Cla}, Theorem~2.39, where it was named ``nonlinear Beurling-Deny criteria''. Given this equivalence to a known criterion, we can then establish the connection to the main setting.

\begin{theorem}\label{thm-BP}
 Let $\mathcal{E}\colon L^2(X,m) \to [0,\infty]$ be a proper, lower semicontinuous functional. Then the following assertions are equivalent:

 \begin{itemize}
  \item[(i)] For all increasing normal contractions $p\colon \IR\to\IR$ and all $u,v\in L^2(X,m)$, the functional $\mathcal{E}$ satisfies
   \begin{align}
    \tag{$\star$} \mathcal{E}(u - p(u-v)) + \mathcal{E}(v + p(u-v)) \leq \mathcal{E}(u) + \mathcal{E}(v).
   \end{align}
  \item[(ii)]  The functional $\mathcal{E}$ has the contraction property. 
  \item[(iii)] $\mathcal{E}$ is a nonlinear Dirichlet form.
 \end{itemize}
\end{theorem}
\begin{proof}
 $(i)\Rightarrow (ii)$: Let $f,g,C$ be given. Set $u = f+g$, $v = f-g$ and define $p\colon \IR\to\IR$ via $p(x):=\tfrac{x}{2} - C(\tfrac{x}{2})$.
 
 This defines an increasing normal contraction as for $x < y$ the condition $0 \leq p(y) - p(x) \leq y-x$ follows from $|C(\tfrac{x}{2}) - C(\tfrac{y}{2})| \leq \tfrac{y-x}{2}$.
 
 Then, $u-v = 2g$ and thus $p(u-v) = g - Cg$. This means $u - p(u-v) = f + Cg$ and $v + p(u-v) = f - Cg$, so $(i)$ turns into $(ii)$.
 
 $(ii)\Rightarrow (i)$: We invert the substitution used in the previous proof. Let $u,v,p$ be given. Set $f = \tfrac{u+v}{2}$, $g = \tfrac{u-v}{2}$ and define $C\colon\IR\to\IR$ via $C(x) := x - p(2x)$.
 
 Here, for all $x < y$, the property $0 \leq p(2y) - p(2x) \leq 2(y-x)$ yields $|C(x) - C(y)|\leq |x-y|$, so $C$ is a normal contraction.
 
 Then, $Cg = \tfrac{u-v}{2} - p(u-v)$ and thus, with $f + Cg = u - p(u-v)$ and $f - Cg = v + p(u-v)$, $(ii)$ turns into $(i)$.
 
 $(i)\Leftrightarrow (iii)$: This is shown in \cite{Cla}, Theorem~2.39.
\end{proof}

\begin{remarks}
 \begin{enumerate}[(a)]
  \item The condition $(i)$ is the contraction property used by \cite{BP, Cla}, while $(ii)$ is our variant.
  \item The above proof shows that the inequality $(\star)$ for a given increasing normal contraction $p$ is equivalent to compatibility with the normal contraction $C$ given by $C(x) := x - p(2x)$. As the proof shows, $(i)\Leftrightarrow (ii)$ holds true regardless of lower semicontinuity.
  \item One can find an analogue of Lemma~\ref{composition} for the inequality $(\star)$. Given $(\star)$ for the increasing normal contractions $p_1, p_2$, we obtain the same inequality for the increasing normal contraction
   $$ \IR\to\IR, \quad x\mapsto p_1(x) + p_2(x - 2p_1(x)). $$
  This was not done in \cite{BP} or \cite{Cla}, presumably because the statement has become rather unwieldy.
 \end{enumerate}
\end{remarks}

\section{Criteria d’après Cipriani / Grillo}

In this section, we will discuss the criterion of Cipriani / Grillo (Definition~\ref{def-main}, from \cite{CG}, Definition~3.1) that was our starting point. Since we have already established equivalence to the contraction property in Theorem~\ref{thm-BP}, we will focus on fitting their criterion into our framework and finding additional characterizations via arguably simpler families of contractions.

The first step will be to rewrite the inequalities appearing in Definition~\ref{def-main} as compatibility with some specific contractions. Taking inspiration from the classical Beurling-Deny criteria \cite{BD, BD1}, we then proceed by exploring some new variants using the resolvent characterization of \cite{CG}, Theorem~3.4. These new characterizations rather closely resemble well-known criteria from the bilinear case.

\begin{theorem}[Compatibility with elementary contractions]\label{thm-CG}
 Let $\mathcal{E}\colon L^2(X,m) \to [0,\infty]$ be a proper, lower semicontinuous, convex functional. Then the following assertions are equivalent:
 
 \begin{itemize}
  \item[(i)] $\mathcal{E}$ is a nonlinear Dirichlet form.
  \item[(ii)] $\mathcal{E}$ is compatible with the set of normal contractions given by
    $$ \IR\to \IR,\quad x \mapsto x_+ = 0\vee x $$
   and
    $$ \IR\to \IR,\quad x\mapsto -\alpha \vee x \wedge \alpha $$
   for all $\alpha > 0$.
  \item[(iii)] $\mathcal{E}$ is compatible with the set of normal contractions given by
    $$ \IR\to \IR,\quad x\mapsto x\wedge \alpha $$
   for all $\alpha \geq 0$.
  \item[(iii)'] $\mathcal{E}$ is compatible with the set of normal contractions given by
    $$ \IR\to \IR,\quad x\mapsto x\wedge \alpha $$
   for all $\alpha > 0$.
  \item[(iv)] $\mathcal{E}$ is compatible with the set of normal contractions given by
    $$ \IR\to \IR,\quad x\mapsto 0 \vee x\wedge \alpha $$
   for all $\alpha > 0$.
 \end{itemize}
\end{theorem}
\begin{proof}
 $(i) \Leftrightarrow (ii)$: Using the substitution $(u,v) = (f+g,f-g)$ as in the proof of Theorem~\ref{thm-BP}, we get
   $$ \frac{f+g + (f+g)\wedge (f-g)}{2} = f - (-g)_+ $$
  and
   $$ \frac{f-g + (f+g)\vee (f-g)}{2} = f + (-g)_+, $$
  as well as
  \begin{align*}
   P_{2,2\alpha}^1(f+g,f-g) & = f-g + \frac{1}{2}[(2g + 2\alpha)_+ - (2g - 2\alpha)_-] \\
                            & = f + (-\alpha) \vee g\wedge \alpha
  \end{align*}
  and
  \begin{align*}
   P_{2,2\alpha}^2(f+g,f-g) & = f+g - \frac{1}{2}[(2g + 2\alpha)_+ - (2g - 2\alpha)_-] \\
                            & = f - (-\alpha) \vee g\wedge \alpha.
  \end{align*}
  The statements in $(i)$ and $(ii)$ are therefore equivalent as they are just different formulations of the same inequalities.
 
 $(ii) \Rightarrow (iv)$: The normal contraction in $(iv)$ can be written as a composition of normal contractions in $(ii)$, as for all $x\in \IR$ we have
   $$ 0\vee x\wedge \alpha = (-\alpha \vee x \wedge \alpha)_+. $$
  
  Thus, using Lemma~\ref{composition} gives the desired result.
 
 $(iv) \Rightarrow (iii)$: This proof is non-elementary. Let $\lambda > 0$. Define the nonlinear resolvent $J_\lambda\colon L^2(X,m) \to L^2(X,m)$ as the (unique) minimizer $J_\lambda f$ of 
   $$ g\mapsto \mathcal{E}(g) + \frac{1}{2\lambda}\|f-g\|_{L^2(X,m)}^2. $$
  The resolvent is $1$-Lipschitz, see \cite{Sho}, Section IV.1.
  
  Furthermore, define the functional $\mathcal{E}^{(2)}\colon L^2(X,m) \times L^2(X,m) \to [0,\infty]$ by
   $$ \mathcal{E}^{(2)}(u,v) = \mathcal{E}(u) + \mathcal{E}(v) $$
  for all $u,v\in L^2(X,m)$. Its resolvent 
   $$ J_\lambda^{(2)}\colon L^2(X,m)\times L^2(X,m) \to L^2(X,m)\to L^2(X,m) $$
  is given by $J_\lambda^{(2)}(u,v) = (J_\lambda u, J_\lambda v)$ for all $u,v\in L^2(X,m)$.
 
  Let $\mathcal{C}\subset L^2(X,m) \times L^2(X,m)$ be a closed and convex subset. Using \cite{CG}, Theorem~3.4 (or rather, its proof), we see that the inequality
  \begin{align}
   \tag{+} \mathcal{E}^{(2)}(\mbox{Proj}_\mathcal{C}(u,v)) \leq \mathcal{E}^{(2)}(u,v)
  \end{align}
  for all $u,v\in L^2(X,m)$ is equivalent to $J_\lambda^{(2)} \mathcal{C} \subseteq \mathcal{C}$ for all $\lambda > 0$.
 
  We rewrite the conditions $(iii)$ and $(iv)$ using these projections, along the lines of \cite{CG}, Lemma~3.3.
 
  The Hilbert projection $P$ onto $\mathcal{C}$ is characterized by the fact that its image is contained in $\mathcal{C}$ and for all $\tilde{f}\in L^2(X,m)\times L^2(X,m)$ and all $\tilde{g}\in \mathcal{C}$, we have
   $$ \langle \tilde{f} - P\tilde{f}, \tilde{g} - P\tilde{f} \rangle \leq 0. $$
  We claim that for
   $$ \mathcal{C} = \{(u,v)\in L^2(X,m)\times L^2(X,m), 2a \leq u-v \leq 2b\} $$
  with $a \leq 0 \leq b$ (these bounds are allowed to be infinite), the projection onto $\mathcal{C}$ is given by
   $$ P(f+g, f-g) = (f+Cg, f-Cg) $$
  for all $f,g\in L^2(X,m)$, where $C \colon \IR\to \IR$, $C(x) = a\vee x \wedge b$. Recall that for all $\tilde{f}\in L^2(X,m)\times L^2(X,m)$ there exist $f,g\in L^2(X,m)$ with $\tilde{f} = (f+g, f-g)$.

  First, it is easy to see that
   $$ (f+Cg, f-Cg) \in \mathcal{C} $$
  for all $f,g\in L^2(X,m)$. Given $\tilde{f}\in L^2(X,m)\times L^2(X,m)$ and $\tilde{g}\in \mathcal{C}$, we can write these functions as $\tilde{f} = (f+g,f-g)$ and $\tilde{g} = (u-v, u+v)$ without loss of generality. Here, we observe that $\tilde{g}\in \mathcal{C}$ is equivalent to $a\leq v\leq b$. Then,
  \begin{align*}
   & \langle \tilde{f} - P\tilde{f}, \tilde{g} - P\tilde{f} \rangle \\
   & = \int (f+g - (f+Cg))(u+v - (f+Cg)) dm \\
   & + \int (f-g - (f-Cg))(u-v - (f-Cg)) dm \\
   & = \int (g - Cg)(u+v - (f+Cg)) dm \\
   & + \int -(g - Cg)(u-v - (f-Cg)) dm \\
   & = 2\int (g - Cg)(v - Cg) dm \leq 0,
  \end{align*}
  as $Cg$ is the Hilbert projection of $g$ onto the convex set $\{f\in L^2(X,m) \mid a \leq f\leq b\}$. Thus, the inequality (+) is precisely the compatibility of $\mathcal{E}$ with the normal contraction $C$.
 
  Using this projection, we see that $(iii)$ is the invariance of the sets
   $$ \{(u,v)\mid u,v\in L^2(X,m), u-v\leq \frac{\alpha}{2}\} $$
  and similarly, $(iv)$ is the invariance of the sets
   $$ \{(u,v)\mid u,v\in L^2(X,m), 0\leq u-v\leq \frac{\alpha}{2}\}. $$
 
  Now, we just have to show the implication of invariance under the resolvent. First, we prove $(iii)$ for $\alpha = 0$. By the preceding discussion, this is equivalent to the statement
   $$ u \geq v \quad \Longrightarrow \quad J_\lambda u \geq J_\lambda v $$
  for all $u,v\in L^2(X,m)$ and all $\lambda > 0$.
  
  Let $u,v\in L^2(X,m)$ with $u\geq v$ be given. For all $M > 0$, we define $u_M := u\wedge (v+M)$. Then, $u_M \to u$ in $L^2(X,m)$ as $M\to\infty$. Using $(iv)$, we get
   $$ 0 \leq u_M - v \leq M \quad \Longrightarrow \quad 0\leq J_\lambda u_M - J_\lambda v \leq M, $$
  so using the continuity of the resolvent yields
   $$ J_\lambda u - J_\lambda v = \lim\limits_{M\to\infty} J_\lambda u_M - J_\lambda v \geq 0. $$
  
  We turn to the remaining case $\alpha > 0$. Let $\lambda > 0$ and let $u,v \in L^2(X,m)$ with $u-v \leq \alpha$ be given. We want to prove $J_\lambda u - J_\lambda v \leq \alpha$.
 
  Define $\tilde{u} := u\vee v$. Then, $0 \leq \tilde{u} - v \leq \alpha$, so by $(iv)$, we get
   $$ 0 \leq J_\lambda \tilde{u} - J_\lambda v \leq \alpha. $$
 
  Using the result of the previous case and $u\leq \tilde{u}$, we obtain
   $$ J_\lambda u - J_\lambda v \leq J_\lambda \tilde{u} - J_\lambda v \leq \alpha. $$
 
 $(iii) \Rightarrow (ii)$: Let $\mathcal{E}$ satisfy $(iii)$ and let $f,g\in L^2(X,m)$. Then, using Lemma~\ref{composition} and the compositions
   $$ x_+ = -((-x)\wedge 0) $$
  and for all $\alpha > 0$
   $$ (-\alpha)\vee x\wedge \alpha = (-((-x)\wedge \alpha))\wedge \alpha $$
  gives the desired results.

 $(iii) \Leftrightarrow (iii)'$: This follows easily from Lemma~\ref{pw-conv}.
\end{proof}

\section{The reflection criterion of Brigati / Hartarsky} 
In this section, we relate the contraction property to a condition established by Brigati and Hartarsky in \cite{BH}, Theorem~1.3.

This theorem was used to prove the classical contraction property $\mathcal{E}(Cf) \leq \mathcal{E}(f)$ (\cite{BH}, Theorem~1.2) for symmetric nonlinear Dirichlet forms. We modify their approach to obtain the contraction property. The idea is that instead of $f\mapsto \mathcal{E}(f)$ we consider a symmetrized variant $g\mapsto \mathcal{E}(f+g) + \mathcal{E}(f-g)$, so we can apply the same proof without requiring symmetry of the original functional $\mathcal{E}$.

For the approximation procedure used in this section, we require the functional to be lower semicontinuous. Recalling that nonlinear Dirichlet forms are all lower semicontinuous, we see that this does not restrict our main setting.

We now state the result of this section. The final part of its proof is rather long, so we split it into separate lemmas.

\begin{theorem}(Reflection criterion)\label{thm-BH}
 Let $\mathcal{E}\colon L^2(X,m)\to [0,\infty]$ be a lower semicontinuous functional. The following assertions are equivalent:
 \begin{itemize}
  \item[(i)] $\mathcal{E}$ has the contraction property.
  \item[(ii)] $\mathcal{E}$ is compatible with the normal contractions
  \begin{align}
   \tag{ii.a} |\cdot|\colon \IR\to\IR, \quad x & \mapsto |x|, \\ 
   \tag{ii.b} C_\alpha\colon \IR\to\IR, \quad x & \mapsto (-\alpha - x) \vee x \wedge (\alpha - x) 
  \end{align}
  for all $\alpha \geq 0$.
  \item[(iii)] For all $f,g\in L^2(X,m)$ and all $\alpha \geq 0$, $\mathcal{E}$ verifies
  \begin{align}
   \tag{iii.a} \mathcal{E}(f\vee g) + \mathcal{E}(f\wedge g) & \leq \mathcal{E}(f) + \mathcal{E}(g), \\
   \tag{iii.b} \mathcal{E}(H_\alpha(f,g)) + \mathcal{E}(H_\alpha(g,f)) & \leq \mathcal{E}(f) + \mathcal{E}(g),
  \end{align}
  where
   $$ H_\alpha(f,g) = (g-\alpha) \vee f \wedge (g+\alpha). $$
 \end{itemize}
\end{theorem}

\begin{remarks}
 \leavevmode
 \vspace{0em}
 \begin{enumerate}[(a)]
  \item The condition $(iii)$ is exactly the condition in \cite{BH}, Theorem~1.3. There is a small omission in the proof, as \cite{BH}, Theorem~2.3 (cited from \cite{Ba}, Proposition~2.5) only needs the twist condition for all $s+t\leq 1$ and it only holds true in this case.
  \item The equivalence of $(ii)$ and $(iii)$ does not require lower semicontinuity and can be split into $(ii.a)\Leftrightarrow(iii.a)$ and $(ii.b)\Leftrightarrow(iii.b)$. This is useful, as it transforms the condition $(iii.a)$ satisfied by order-preserving forms into compatibility with the normal contraction $|\cdot|$.
  \item The name of this section stems from the geometric interpretation of the criteria of the previous section. More specifically, the maps $L^2(X,m)\times L^2(X,m) \to L^2(X,m)\times L^2(X,m)$ of Cipriani / Grillo are exactly the arithmetic mean of the identity map and the maps used by Brigati / Hartarsky, i.e.
   $$ (\frac{f + f\wedge g}{2}, \frac{g + f\vee g}{2}) = \frac{(f,g) + (f\wedge g, f\vee g)}{2} $$
  and
   $$ P_{2,2\alpha}(f,g) = \frac{(f,g) + H_\alpha(f,g)}{2}. $$
  
  In that sense, one can interpret the criterion of Brigati / Hartarsky \cite{BH} as a statement on the reflections with respect to the projections in Cipriani / Grillo \cite{CG}.
  \item We repeat the proof here to close a small gap in the original proof. Fortunately, this does not require any new ideas.
 \end{enumerate}
\end{remarks}

Most steps of the proof will be rather direct, with the sole exception of the implication $(ii)/(iii) \Rightarrow (i)$. This difficulty motivates taking a closer look at the functionals satisfying $(ii)$, with the goal of proving the contraction property. We thus (mostly) follow in the footsteps of \cite{BH}, Theorem~1.2. The central idea is that the problem can be reduced to contractions of bounded complexity via composition.

We start by proving convexity. This will be useful to show finiteness of some terms we want to subtract from certain inequalities.
 
\begin{lemma}[Convexity]\label{convex}
 Let $\mathcal{E}\colon L^2(X,m)\to [0,\infty]$ be a lower semicontinuous functional satisfying $(ii)$ of Theorem~\ref{thm-BH}, i.e. for all $f,g\in L^2(X,m)$ and all $\alpha \geq 0$, $\mathcal{E}$ is compatible with the normal contractions
  \begin{align}
   \tag{ii.a} |\cdot|\colon \IR\to\IR, \quad x & \mapsto |x|, \\ 
   \tag{ii.b} C_\alpha\colon \IR\to\IR, \quad x & \mapsto (-\alpha - x) \vee x \wedge (\alpha - x) 
  \end{align}
  for all $\alpha \geq 0$.
  Then $\mathcal{E}$ is convex.
\end{lemma}
\begin{proof}
 By standard arguments, any lower semicontinuous functional $\mathcal{E}$ satisfying
  $$ \mathcal{E}\left(\frac{f+g}{2}\right) \leq \frac{\mathcal{E}(f) + \mathcal{E}(g)}{2} $$
 for all $f,g\in L^2(X,m)$ must be convex.
 
 Setting $\tilde{f} = \frac{f+g}{2}$ and $\tilde{g} = \frac{f-g}{2}$, we see that convexity is equivalent to compatibility with the vanishing normal contraction $\IR\to\IR, x\mapsto 0$. This normal contraction can be written as a pointwise limit $\lim\limits_{n\to \infty} D_n(x) = 0$ for all $x\in \mathbb{R}$, where the approximants $D_n$ are obtained via composition as
  $$ D_n := C_{2\cdot 3^{-n}}\circ C_{2 \cdot3^{1-n}}\circ \ldots \circ C_{2\cdot3^{n-1}} \circ C_{2\cdot3^n}. $$
 Pointwise convergence follows from the fact that $C_{2\alpha}$ maps the interval $[-3\alpha, 3\alpha]$ to the interval $[-\alpha, \alpha]$, so each normal contraction $D_n$ will map the interval $[-3^{n+1},3^{n+1}]$ into the interval $[-3^{-n}, 3^{-n}]$.
 
 Compatibility with each $D_n$ can be shown by Lemma~\ref{composition} and Lemma~\ref{pw-conv} finishes the proof.
\end{proof}

The next lemma performs a reduction step. As we did not change anything concerning its setting, it remains an exact replica of \cite{BH}, Lemma~1.4, where its proof can be found.
\begin{lemma}[Reduction]\label{reduction}
 Let $\Phi$ be the set of all normal contractions.

 Let $G$ be the set of all normal contractions $\phi\in \Phi$ such that $|\phi'| = 1$ almost everywhere and $\phi'$ has at most two points of discontinuity. Let $\langle G\rangle$ be the collection of all finite compositions of elements in $G$. Then, $\langle G\rangle$ is dense in $\Phi$ for the pointwise convergence on $\IR$. 
\end{lemma}

Together with Lemma~\ref{pw-conv}, this means we only need to prove compatibility with contractions in $G$. Details of this reduction step are given in the proof of the theorem.

The contractions in $G$ are of the form $\pm id$ (no sign changes, trivial), $\pm \phi_x$ (one sign change at $x\in \IR$, slope starting at $+1$) and $\pm \phi_{x_1, x_2}$ (two sign changes at $x_1 < x_2$, slope starting at $+1$).

This leads to three interesting cases that will be the subject of one lemma, respectively.

\begin{lemma}[First case]\label{case1}
 Let $\mathcal{E}\colon L^2(X,m)\to [0,\infty]$ be a lower semicontinuous functional satisfying $(ii)$ of Theorem~\ref{thm-BH}, i.e. $\mathcal{E}$ is compatible with the normal contractions
 \begin{align}
  \tag{ii.a} |\cdot|\colon \IR\to\IR, \quad x & \mapsto |x|, \\ 
  \tag{ii.b} C_\alpha\colon \IR\to\IR, \quad x & \mapsto (-\alpha - x) \vee x \wedge (\alpha - x) 
 \end{align}
 for all $\alpha \geq 0$.
  
 Then, $\mathcal{E}$ is compatible with the normal contractions
  $$ \phi_x\colon \IR\to\IR, \quad \phi_x(t) = \begin{cases} t - 2(t-x)_+, & x\geq 0 \\ -t - 2(x-t)_+, & x < 0 \end{cases} $$
 for all $x\in \IR$.
\end{lemma}
\begin{proof}
 Let $x \geq 0$ be a threshold. We will show compatibility with this contraction via its companion $\sigma_x\colon t \mapsto \phi_x(t_+)$. Using the fact that $\phi_x(t) + t_+ = \sigma_x(t) + t$ and $t_+ - \phi_x(t) = |\sigma_x(t) - t|$,
 we get
  $$ \mathcal{E}(f + \phi_x g) = \mathcal{E}(f + \tfrac{\sigma_x g + g}{2} - \tfrac{|\sigma_x g - g|}{2}) $$
 and
  $$ \mathcal{E}(f + g_+) = \mathcal{E}(f + \tfrac{\sigma_x g + g}{2} + \tfrac{|\sigma_x g - g|}{2}). $$
 The condition $(ii.a)$ implies
  $$ \mathcal{E}(f + \phi_x g) + \mathcal{E}(f + g_+) \leq \mathcal{E}(f+g) + \mathcal{E}(f + \sigma_x g). $$
 In a similar vein, we obtain
  $$ \mathcal{E}(f - \phi_x g) + \mathcal{E}(f - g_+) \leq \mathcal{E}(f-g) + \mathcal{E}(f - \sigma_x g). $$
 Since $\sigma_x(t) = C_x(t_+)$, the condition $(ii.b)$ implies
  $$ \mathcal{E}(f + \sigma_x g) + \mathcal{E}(f - \sigma_x g) \leq \mathcal{E}(f + g_+) + \mathcal{E}(f - g_+). $$
 Adding these three inequalities yields an inequality that is the desired one plus some terms on both sides. In order to be able to subtract these extraneous terms, we have to prove their finiteness under the assumption that $\mathcal{E}(f+g) + \mathcal{E}(f-g)$ is finite. This can be done via Lemma~\ref{convex}, i.e.
  $$ \mathcal{E}(f \pm g_+) \leq \frac{1}{2}(\mathcal{E}(f \pm g) + \mathcal{E}(f \pm |g|)) $$
 and the condition $(ii.a)$ imply
  $$ \mathcal{E}(f + g_+) + \mathcal{E}(f - g_+) \leq \mathcal{E}(f + g) + \mathcal{E}(f - g) < \infty. $$
 Applying the third inequality proven here shows the finiteness of $\mathcal{E}(f\pm \sigma_x g)$. As these terms are finite, we can subtract them from both sides of the summed inequality, proving the claim
  $$ \mathcal{E}(f + \phi_x g) + \mathcal{E}(f - \phi_x g) \leq \mathcal{E}(f + g) + \mathcal{E}(f - g). $$
 
 The case $x < 0$ admits basically the same proof.
\end{proof}

\begin{remarks}
 Here we have the first (somewhat) substantial difference to the proof of \cite{BH}, Theorem~1.2. In the original, the finiteness of the extraneous terms was not shown. This small gap should not detract from \cite{BH}.
\end{remarks}

For the functions $\phi_{x_1, x_2}$, we distinguish two cases as the proof depends on whether $0$ is contained in the interval $(x_1, x_2)$ or not.
 
\begin{lemma}[Second case]\label{case2}
 Let $x_1, x_2 \in \IR$ satisfy $0 \leq x_1 < x_2$ or $x_1 < x_2 \leq 0$.
 
 Let $\mathcal{E}\colon L^2(X,m)\to [0,\infty]$ be a lower semicontinuous functional satisfying $(ii)$ of Theorem~\ref{thm-BH}, i.e. $\mathcal{E}$ is compatible with the normal contractions
 \begin{align}
  \tag{ii.a} |\cdot|\colon \IR\to\IR, \quad x & \mapsto |x|, \\ 
  \tag{ii.b} C_\alpha\colon \IR\to\IR, \quad x & \mapsto (-\alpha - x) \vee x \wedge (\alpha - x) 
 \end{align}
 for all $\alpha \geq 0$.
 
 Then, $\mathcal{E}$ is compatible with the normal contraction $\phi_{x_1, x_2}\colon \IR\to\IR$,
  $$ \phi_{x_1, x_2}(t) = \begin{cases} t - 2(t - x_1)_+ + 2(t - x_2)_+, & 0\leq x_1 < x_2 \\ t + 2(x_2 - t)_+ - 2(x_1 - t)_+, & x_1 < x_2 \leq 0. \end{cases} $$
\end{lemma}
\begin{proof}
 Assume $0 \leq x_1 < x_2$. We need the auxiliary contractions
  $$ \sigma(x) = (0 \wedge (x_1 - x)) \vee (x + x_1 - 2x_2) $$
 and
  $$ \psi(x) = \phi_{x_1,x_2}(x_+). $$

 This time, we will need to sum up five inequalities. The first inequality rests on the fact that $C_{x_1} (x_+ - \sigma(x)) = \psi(x) - (x - x_1)_+$, implying
  $$ \mathcal{E}(f + \psi g) + \mathcal{E}(f + (g - x_1)_+) \leq \mathcal{E}(f + g_+) + \mathcal{E}(f + \sigma g) $$
 as well as
  $$ \mathcal{E}(f - \psi g) + \mathcal{E}(f - (g - x_1)_+) \leq \mathcal{E}(f - g_+) + \mathcal{E}(f - \sigma g). $$
  
 Next, $C_{x_2 - x_1} ((id - x_1)_+) = -\sigma$ proves the inequality
  $$ \mathcal{E}(f - \sigma g) + \mathcal{E}(f + \sigma g) \leq \mathcal{E}(f + (g - x_1)_+) + \mathcal{E}(f - (g - x_1)_+). $$
  
 The remaining pair of inequalities relies on the condition $(ii.a)$, where $|\psi(x) - x| = x_+ - \phi_{x_1, x_2}(x)$ shows
  $$ \mathcal{E}(f + g_+) + \mathcal{E}(f + \phi_{x_1, x_2} g) \leq \mathcal{E}(f + \psi g) + \mathcal{E}(f + g) $$
 and
  $$ \mathcal{E}(f - g_+) + \mathcal{E}(f - \phi_{x_1, x_2} g) \leq \mathcal{E}(f - \psi g) + \mathcal{E}(f - g). $$

 Summing these five inequalities yields another inequality that has some unwanted terms. Again, we have to prove finiteness of these terms under the assumption that $\mathcal{E}(f+g) + \mathcal{E}(f-g)$ is finite. We recall that the finiteness discussions in the proof of Lemma~\ref{case1} already state that
  $$ \mathcal{E}(f + g_+) + \mathcal{E}(f - g_+) \leq \mathcal{E}(f + g) + \mathcal{E}(f - g). $$
 We move on to the next term, where
  \begin{align*}
   \mathcal{E}(f \pm (g - x_1)_+) & \leq \frac{1}{2}(\mathcal{E}(f \pm g_+) + \mathcal{E}(f \mp \phi_{x_1}g_+))
  \end{align*}
 by convexity of $\mathcal{E}$. Combining this with our observations from Lemma~\ref{case1}, we obtain
  $$ \mathcal{E}(f + (g - x_1)_+) + \mathcal{E}(f - (g - x_1)_+) \leq \mathcal{E}(f + g) + \mathcal{E}(f - g). $$
 The remaining terms $\mathcal{E}(f\pm \sigma g)$ and $\mathcal{E}(f\pm \psi g)$ are shown to be finite by the first three inequalities, so we can subtract them, obtaining the inequality
  $$ \mathcal{E}(f + \phi_{x_1, x_2} g) + \mathcal{E}(f - \phi_{x_1, x_2} g) \leq \mathcal{E}(f + g) + \mathcal{E}(f - g). $$

 As in Lemma~\ref{case1}, the case $x_1 < x_2\leq 0$ admits a mostly identical proof.
\end{proof}
 
We now move on to the final case.
 
\begin{lemma}[Third case]\label{case3}
 Let $x_1 < 0 < x_2$.
 
 Let $\mathcal{E}\colon L^2(X,m)\to [0,\infty]$ be a lower semicontinuous functional satisfying $(ii)$ of Theorem~\ref{thm-BH}, i.e. $\mathcal{E}$ is compatible with the normal contractions
 \begin{align}
  \tag{ii.a} |\cdot|\colon \IR\to\IR, \quad x & \mapsto |x|, \\ 
  \tag{ii.b} C_\alpha\colon \IR\to\IR, \quad x & \mapsto (-\alpha - x) \vee x \wedge (\alpha - x) 
 \end{align}
 for all $\alpha \geq 0$.
 
 Then, $\mathcal{E}$ is compatible with the normal contraction
  $$ \phi_{x_1, x_2}\colon \IR\to\IR, \quad \phi_{x_1, x_2}(t) = -t - 2(t - x_1)_+ + 2(t - x_2)_+. $$
\end{lemma}
\begin{proof}
 This case is dealt with by introducing the intermediary normal contraction
  $$ \psi(x) := (x - 2x_1) \wedge -(x \wedge x_2). $$
 The observation $\psi = \phi_{x_1} \circ (id \wedge x_2)$ together with the already known compatibility with the normal contractions $\phi_x$ (Lemma~\ref{case1}) implies 
  $$ \mathcal{E}(f + \psi g) + \mathcal{E}(f - \psi g) \leq \mathcal{E}(f + (g\wedge x_2)) + \mathcal{E}(f - (g\wedge x_2)). $$
 Another application of the compatibility with $\phi_x$ in the guise of
  $$ \phi_{2x_2} (id - \psi) = (id\wedge x_2 - \phi_{x_1,x_2}) $$
 yielding
  $$ \mathcal{E}(f + \phi_{x_1,x_2} g) + \mathcal{E}(f + (g \wedge x_2)) \leq \mathcal{E}(f + g) + \mathcal{E}(f + \psi g) $$
 and
  $$ \mathcal{E}(f - \phi_{x_1,x_2} g) + \mathcal{E}(f - (g \wedge x_2)) \leq \mathcal{E}(f - g) + \mathcal{E}(f - \psi g) $$
 again gives us (in total) three inequalities that sum up to almost the desired statement. Once more, we want to get rid of the equal terms on both sides, so we assume $\mathcal{E}(f + g) + \mathcal{E}(f - g) < \infty$ and use convexity to obtain
  $$ \mathcal{E}(f \pm (g\wedge x_2)) \leq \frac{1}{2}(\mathcal{E}(f \pm g) + \mathcal{E}(f \pm \phi_{x_2} g)), $$
 which we combine with Lemma~\ref{case1} to prove
  $$ \mathcal{E}(f + (g\wedge x_2)) + \mathcal{E}(f - (g\wedge x_2)) \leq \mathcal{E}(f + g) + \mathcal{E}(f - g). $$
 The other term
  $$ \mathcal{E}(f + \psi g) + \mathcal{E}(f - \psi g) \leq \mathcal{E}(f + (g\wedge x_2)) + \mathcal{E}(f - (g\wedge x_2)) $$
 is handled by the first inequality. Subtracting all of these terms yields the desired result.
\end{proof}

Now, we have everything we need to assemble the proof of the theorem.

\begin{proof}[Proof of Theorem~\ref{thm-BH}]
 $(ii) \Leftrightarrow (iii)$: The contraction $|\cdot|$ of $(ii.a)$ corresponds to $(iii.a)$ via
  $$ (f\vee g, f\wedge g) = \frac{f+g}{2} \pm \left|\frac{f-g}{2}\right|. $$

 Similarly, the contractions $C_\alpha$ of $(ii.b)$ correspond to $(iii.b)$ via
  $$ (H_\alpha(f,g), H_\alpha(g,f)) = \frac{f+g}{2} \pm C_\alpha\left(\frac{f-g}{2}\right). $$
  
 $(i) \Rightarrow (ii)$: This is clear.
 
 $(iii) \Rightarrow (i)$: Let $\mathcal{E}\colon L^2(X,m) \to [0,\infty]$ be lower semicontinuous and satisfy $(iii)$, i.e.
   \begin{align*}
    \mathcal{E}(f\vee g) + \mathcal{E}(f\wedge g) & \leq \mathcal{E}(f) + \mathcal{E}(g), \\
    \mathcal{E}(H_\alpha(f,g)) + \mathcal{E}(H_\alpha(g,f)) & \leq \mathcal{E}(f) + \mathcal{E}(g)
   \end{align*}
 for all $f,g\in L^2(X,m)$ and all $\alpha \geq 0$. We want to show that $\mathcal{E}$ has the contraction property.

 By Lemma~\ref{convex}, $\mathcal{E}$ must be a convex functional.

 Obviously, compatibility with normal contractions is preserved by compositions as we can just apply the compatibility with each contraction in sequence. As we saw in Lemma~\ref{pw-conv}, compatibility with contractions is also preserved under pointwise limits of these normal contractions.
 
 Thus, according to Lemma~\ref{reduction}, we only have to prove compatibility with the contractions in $G$, i.e. all $\phi\colon \IR \to \IR$ with $|\phi'| = 1$ for all but at most two points (this discussion is functionally identical to the arguments at the end of the proof of \cite{BH}, Theorem 1.2).
 
 These contractions are taken care of by Lemma~\ref{case1} (corresponding to \cite{BH}, Proposition~3.1), Lemma~\ref{case2} (\cite{BH}, Proposition~3.2) and Lemma~\ref{case3} (\cite{BH}, Proposition~3.3).
\end{proof}

\section{Summary}

We now collect all of the results. For the compatibility with specific contractions, we only pick the weakest condition.

Also, for the sake of elegance, we use the interpretation via the $f$-shift, see Remarks \ref{rem-contraction}.

\begin{coro}
 Let $\mathcal{E}\colon L^2(X,m) \to [0,\infty]$ be a proper, lower semicontinuous, convex functional. For all $f\in L^2(X,m)$ with $\mathcal{E}(f) < \infty$, define $\mathcal{E}_f\colon L^2(X,m) \to [0,\infty]$ by
  $$ \mathcal{E}_f(g) := \frac{1}{2}(\mathcal{E}(f+g) + \mathcal{E}(f-g)) - \mathcal{E}(f). $$

 Then the following assertions are equivalent:
 \begin{itemize}
  \item[(i)] $\mathcal{E}$ is a nonlinear Dirichlet form.
  \item[(ii)] $\mathcal{E}$ verifies
    $$ \mathcal{E}_f(Cg) \leq \mathcal{E}_f(g) $$
   for all $f\in D(\mathcal{E})$, $g\in L^2(X,m)$ and all normal contractions $C\colon \IR \to \IR$.
  \item[(iii)] $\mathcal{E}$ verifies
    $$ \mathcal{E}_f(0 \vee g\wedge \alpha) \leq \mathcal{E}_f(g) $$
   for all $f\in D(\mathcal{E})$, $g\in L^2(X,m)$ and all $\alpha \geq 0$.
 \end{itemize}

 If $\mathcal{E}$ is positively $p$-homogeneous (i.e. $\mathcal{E}(\alpha f) = \alpha^p f$ for all $\alpha > 0$) for some $p\geq 1$, then $(iii)$ can be reduced to the case $\alpha = 1$, i.e. 
  $$ \mathcal{E}_f(0\vee g \wedge 1) \leq \mathcal{E}_f(g) $$
 for all $f\in D(\mathcal{E})$ and all $g\in L^2(X,m)$.
\end{coro}
\begin{proof}
 Only the last remark is new. For this result, we (mostly) reprise the proof of \cite{CG}, Corollary~3.7.
 
 Let $\alpha > 0$. It is easy to see that $\mathcal{E}(f) < \infty$ implies $\mathcal{E}(\alpha f) < \infty$ and
  $$ \mathcal{E}_{\alpha f}(\alpha g) = \alpha^p \mathcal{E}_f(g). $$
 Thus,
  $$ \mathcal{E}_f(0\vee g\wedge \alpha) = \alpha^p\mathcal{E}_{f/\alpha}(0\vee \tfrac{g}{\alpha}\wedge 1) \leq \alpha^p\mathcal{E}_{f/\alpha}(\tfrac{g}{\alpha}) = \mathcal{E}_f(g). $$
\end{proof}

\end{document}